\documentclass[11pt,a4paper,reqno]{amsart}%
\usepackage{amsthm,amsmath,amsfonts,amssymb,amsxtra,appendix,bookmark,euscript,delarray,dsfont}
\usepackage[latin1]{inputenc}

\makeatletter
\def\@tocline#1#2#3#4#5#6#7{\relax
  \ifnum #1>\c@tocdepth 
  \else
    \par \addpenalty\@secpenalty\addvspace{#2}%
    \begingroup \hyphenpenalty\@M
    \@ifempty{#4}{%
      \@tempdima\csname r@tocindent\number#1\endcsname\relax
    }{%
      \@tempdima#4\relax
    }%
    \parindent\z@ \leftskip#3\relax \advance\leftskip\@tempdima\relax
    \rightskip\@pnumwidth plus4em \parfillskip-\@pnumwidth
    #5\leavevmode\hskip-\@tempdima
      \ifcase #1
       \or\or \hskip 1em \or \hskip 2em \else \hskip 3em \fi%
      #6\nobreak\relax
      \dotfill
      \hbox to\@pnumwidth{\@tocpagenum{#7}}
    \par
    \nobreak
    \endgroup
  \fi}
\makeatother

\theoremstyle{plain}
\newtheorem{theorem}{Theorem}[section]
\newtheorem{lemma}[theorem]{Lemma}

\newtheorem{proposition}[theorem]{Proposition}

\theoremstyle{remark}
\newtheorem{remark}[theorem]{Remark}

\numberwithin{equation}{section}

\newcommand{\ii}{\infty}
\newcommand\R{{\ensuremath {\mathbb R} }}

\newcommand\nn{\nonumber}

\renewcommand\phi{\varphi}

\newcommand{\gH}{\mathfrak{H}}

\newcommand{\wto}{\rightharpoonup}

\newcommand{\cQ}{\mathcal{Q}}

\newcommand{\cE}{\mathcal{E}}

\newcommand{\cF}{\mathcal{F}}

\newcommand{\cL}{\mathcal{L}}
\newcommand{\eps}{\epsilon}

\renewcommand{\epsilon}{\varepsilon}
\newcommand\pscal[1]{{\ensuremath{\left\langle #1 \right\rangle}}}
\newcommand{\norm}[1]{ \left| \! \left| #1 \right| \! \right| }
\DeclareMathOperator{\tr}{{\rm Tr}}
\DeclareMathOperator{\Tr}{{\rm Tr}}

\renewcommand{\ge}{\geqslant}
\renewcommand{\le}{\leqslant}
\renewcommand{\geq}{\geqslant}
\renewcommand{\leq}{\leqslant}
\renewcommand{\hat}{\widehat}
\renewcommand{\tilde}{\widetilde}

\title[Blow-up of rotating 2D focusing Bose gases]{Blow-up profile of rotating 2D focusing Bose gases}

\author[M. Lewin]{Mathieu LEWIN}
\address{CNRS \& CEREMADE, Universit\'e Paris-Dauphine, PSL Research University, F-75016 Paris, France} 
\email{mathieu.lewin@math.cnrs.fr}

\author[P.~T. Nam]{Phan Th\`anh NAM}
\address{Ludwig Maximilian University of Munich, Department of Mathematics, Theresienstrasse 39, D-80333 Munich, Germany} 
\email{nam@math.lmu.de}

\author[N. Rougerie]{Nicolas ROUGERIE}
\address{Universit\'e Grenoble-Alpes \& CNRS,  LPMMC (UMR 5493), B.P. 166, F-38042 Grenoble, France}
\email{nicolas.rougerie@grenoble.cnrs.fr}

\begin{document}
\date{Revised version from April, 2018. First version in February, 2018}

\begin{abstract} 
We consider the Gross-Pitaevskii equation describing an attractive Bose gas trapped to a quasi 2D layer by means of a purely harmonic potential, and which rotates at a fixed speed of rotation $\Omega$. First we study the behavior of the ground state when the coupling constant approaches $a_*$, the critical strength of the cubic nonlinearity for the focusing nonlinear Schrödinger equation. We prove that blow-up always happens at the center of the trap, with the blow-up profile given by the Gagliardo-Nirenberg solution. In particular, the blow-up scenario is independent of $\Omega$, to leading order. This generalizes results obtained by Guo and Seiringer (\textit{Lett. Math. Phys.}, 2014, vol. 104, p. 141--156) in the non-rotating case. In a second part we consider the many-particle Hamiltonian for $N$ bosons, interacting with a potential rescaled in the mean-field manner $-a_NN^{2\beta-1}w(N^\beta x)$, with $w\geq0$ a positive function such that $\int_{\R^2}w(x)\,dx=1$. Assuming that $\beta<1/2$ and that $a_N\to a_*$ sufficiently slowly, we prove that the many-body system is fully condensed on the Gross-Pitaevskii ground state in the limit $N\to\infty$. 
\end{abstract}

\maketitle

\bigskip\bigskip

\begin{flushright}
\sl Dedicated to Herbert Spohn, on the occasion of his 70th birthday
\end{flushright}

\bigskip\bigskip

\setcounter{tocdepth}{2}
\tableofcontents

\section{Introduction}

Because of their ability to display quantum effects at the macroscopic scale, Bose-Einstein condensates (BEC) have become an important subject of research, in particular after their first realization in the laboratory in 1995~\cite{CorWie-95,CorWie-02,Ketterle-95,Ketterle-02}. Condensates with \emph{attractive} interactions are expected to behave quite differently from the better understood \emph{repulsive} case, and they have generated many experimental, numerical or theoretical works. Some atoms like $^7$Li indeed have a negative scattering length and were initially believed not to be able to form a condensate, until attractive BECs were finally experimentally realized in traps~\cite{BraSacTolHul-95}. 

For attractive interactions, the Gross-Pitaevskii functional, commonly used to describe BECs, predicts a collapse of the system when $Na$ (the number of particles times the scattering length) is too negative~\cite{BayPet-96,DalStr-96,UedLeg-98,MueBay-00,GuoSei-14}, an effect which has been observed in some experiments~\cite{GerStrProHul-00,DonClaCorRobCorWie-01}. In addition, attractive Bose-Einstein condensates are believed to respond to rotation in a rather different manner from the repulsive case. In a rotating repulsive Bose gas, a triangular lattice of vortices is formed, with the number of vortices increasing with the speed of rotation~\cite{Cooper-08,Fetter-09,Aftalion-06,CorPinRouYng-11b}. On the contrary, it has been argued~\cite{WilGunSmi-98,Mottelson-99,PetPit-00,SaiUed-04,LunColSuo-04,ColLunSuo-05,CarCla-06,SakMal-08} that in an attractive rotating Bose gas, vortices should be unstable and it is instead the center of mass of the system which can rotate around the axis. 

In this paper, we rigorously establish two results about 2D attractive Bose-Einstein condensates in the critical regime of collapse. We consider a Bose gas trapped to a quasi 2D layer by means of a purely harmonic potential, and which rotates at a fixed speed of rotation $\Omega$. First we look at the Gross-Pitaevskii equation which describes the macroscopic behavior of the condensate~\cite{BaoCai-13}. We study its solutions in the regime where the coupling constant approaches the critical blow-up value $a_*$, given by the best constant in the Gagliardo-Nirenberg inequality. In this case we prove that blow-up always happens at the center of the trap, with the blow-up profile given by the Gagliardo-Nirenberg optimizer. This shows that the rotation does not affect the general blow-up scenario, to leading order. The non-rotating case has been previously studied by Guo and Seiringer in~\cite{GuoSei-14}. Other similar mathematical results on the trapped nonlinear Schrödinger equation for non-rotating gases include~\cite{Maeda-10,DenGuoLu-15,GuoZenZho-16,Thanh-17} in the stationary case and~\cite{Zhang-00,Carles-02,Zhang-05} in the time-dependent case.

In a second part we consider the many-particle (microscopic) Hamiltonian for $N$ such bosons, interacting with a potential rescaled in the mean-field manner
$$-a_NN^{2\beta-1}w(N^\beta x)$$
with $w\geq0$ a fixed positive function such that $\int_{\R^2}w(x)\,dx=1$. Assuming that $0<\beta<1/2$ and that $a_N\to a_*$ sufficiently slowly, we are able to show that the many-body system is fully condensed on the Gross-Pitaevskii ground state studied in the first step, in the limit $N\to\ii$. This justifies the validity of the Gross-Pitaevskii equation in this regime of collapse, with complete Bose-Einstein condensation at the point of blow-up. We do not observe fragmented condensation at this order. 

The mathematical method used here follows several of our previous papers~\cite{LewNamRou-14,LewNamRou-15b,LewNamRou-16c,LewNamRou-17b}. Note that some authors have already dealt with the time-dependent equation, in the subcritical regime $a<a_*$, see~\cite{CheHol-17,LewNamRou-17b,JebPic-17,NamNap-17}.

The next section contains the precise definition of our model as well as the statement of our main results. The remainder of the paper is then devoted to their proofs. In Appendix~\ref{sec:extensions} we mention several possible extensions of our findings, without giving the detailed  mathematical proofs.

\subsubsection*{\bf Acknowledgement} 
It is our pleasure to dedicate this paper to Herbert Spohn, on the occasion of his 70th birthday. This project has received funding from the European Research Council (ERC) under the European Union's Horizon 2020 research and innovation programme (grant agreements MDFT No 725528 and CORFRONMAT No 758620).

\section{Models and main results}

\subsection{Collapse of the rotating Gross-Pitaevskii ground state}

The Gross-Pitaevskii functional describing a condensed system of bosons trapped to a 2D plane and rotating along the third axis at an angular velocity $\Omega$ reads
\begin{equation}
\cE^{\rm GP}_{\Omega,a}(u)=\int_{\R^2}|\nabla u(x)|^2\,dx+\int_{\R^2}|x|^2|u(x)|^2\,dx-2\Omega\pscal{u,Lu}-\frac{a}2\int_{\R^2}|u(x)|^4\,dx
\label{eq:def_energy}
\end{equation}
where $L=-i x \wedge \nabla=i(x_2\partial_1-x_1\partial_2)$ is the angular momentum. Here we have chosen units such that the trapping potential has a trapping frequency $\Omega_{\rm trap}=1$, for simplicity. The system is stable under the assumption that $|\Omega|<1$ and $a<a_*$, where $a_*$ is the optimal constant of the Gagliardo-Nirenberg inequality~\cite{Weinstein-83,Zhang-00,GuoSei-14,Maeda-10,Frank-13}
\begin{align} \label{eq:GN}
 \left(\int_{\R^2} |\nabla u|^2 \right) \left(\int_{\R^2} |u|^2 \right) \ge \frac{a_*}{2} \int_{\R^2} |u|^4, \quad \forall u\in H^1(\R^2).
\end{align}
Equivalently,  
$$a_* = \norm{Q}_{L ^2 (\R ^2)} ^2, $$ 
where $Q\in H ^1 (\R ^2)$ is the unique positive solution of 
\begin{equation}\label{eq:Q-GN}
-\Delta Q + Q - Q ^3 = 0 \mbox{ in } \R ^2,
\end{equation}
up to translations. More precisely, $Q$ is symmetric radial decreasing and it is the unique (up to translations and dilations) optimizer for the Gagliardo-Nirenberg inequality \eqref{eq:GN}. 
In the following we therefore always assume that $0\leq \Omega<1$ and $0<a<a_*$. The energy may also be written in the form
\begin{multline}
 \cE^{\rm GP}_{\Omega,a}(u)=\int_{\R^2}|\nabla u(x)+i\Omega x^\perp u(x)|^2\,dx+(1-\Omega^2)\int_{\R^2}|x|^2|u(x)|^2\,dx\\-\frac{a}2\int_{\R^2}|u(x)|^4\,dx
 \label{eq:def_energy_carre}
\end{multline}
where $x^\perp=(-x_2,x_1)$. We call
$$\boxed{E^{\rm GP}_\Omega(a):=\min_{\substack{u\in H^1(\R^2), \norm{u}_{L^2}=1}}\cE^{\rm GP}_{\Omega,a}(u)}$$
the ground state energy and look at the limit $a\to a_*$ at fixed $0\leq\Omega<1$. The existence of ground states follows the standard direct method in the calculus of variations. Our first main result is the following.

\begin{theorem}[\textbf{Collapse of rotating Gross-Pitaevskii ground states}]\label{thm:GP}\mbox{}\\
Let $0\leq \Omega<1$ be any fixed rotation. Then we have
\begin{equation}
\boxed{E^{\rm GP}_\Omega(a)=E^{\rm GP}_0(a)+o\big(E^{\rm GP}_0(a)\big)=\sqrt{a_*-a} \left( \frac{2\lambda_*^2}{a_*} + o(1)\right),}
 \label{eq:convergence_energy}
\end{equation}
when $a\to a_*$, where
$$
\lambda_*=\left( \int_{\R^2} |x|^2 |Q(x)|^2 dx \right)^{\frac{1}{4}}
$$
and $Q$ is the unique radial positive Gagliardo-Nirenberg solution~\eqref{eq:Q-GN}.

In addition, for any sequence $a_N\to a_*$ and any sequence $\{u_N\}$ such that $\|u_N\|_{L^2(\R^2)}=1$ and 
\begin{equation}
\cE^{\rm GP}_{\Omega,a_N}(u_N)=E^{\rm GP}_\Omega(a_N)+o(\sqrt{a_*-a_N}),
\label{eq:def_approx_GS}
\end{equation}
(for instance $u_N$ a minimizer of $E_\Omega(a_N)$), we have 
\begin{equation}
\boxed{\lim_{N\to\ii}\norm{u_{N}-e^{i\theta_N}Q_{N}}_{L^2(\R^2)}=0,}
\label{eq:approx_GS}
\end{equation}
for a properly chosen phase $\theta_N\in [0,2\pi)$, where
\begin{equation}
 Q_{N}(x)=  (a_*)^{-1/2}\lambda_* (a_*-a_N)^{-\frac{1}{4}} Q\left( \lambda_* (a_*-a_N)^{-\frac{1}{4}} x\right)
 \label{eq:def_Q_N}
\end{equation}
is the rescaled Gagliardo-Nirenberg optimizer which blows up at the origin at speed $ (a_*-a_N)^{\frac{1}{4}}/\lambda_*$.
\end{theorem}

This theorem was proved by Guo and Seiringer in~\cite{GuoSei-14} in the case $\Omega=0$, with the convergence of the ground states but not of general ``approximate ground states" $u_N$'s as in~\eqref{eq:approx_GS}. Our theorem shows that the results found by Guo and Seiringer remain valid when the system is set in rotation, the blow-up scenario being independent of $\Omega$ to leading order. In particular, we do not see a rotation of the center of mass at this order. Our method of proof relies on the non-degeneracy of the minimizer $Q$, which is known to play a fundamental role in many situations~\cite{Weinstein-85,ChaGusNakTsa-08}. We expect that the non-degeneracy of $Q$ should provide quantitative estimates for the difference between $Q_N$ and the ground state of $E_\Omega(a_N)$, however we have not investigated this question in details.

\begin{remark}\mbox{}\\
(i) It would be interesting to investigate the case where $\Omega=\Omega_N\to1$ at the same time as $a_N\to a_*$. In this case the centrifugal force almost compensates the trapping potential, and this effect could compete with the collapse scenario induced by attractive interactions. 

\smallskip

\noindent (ii) Our proof covers more general external potentials attaining their minimum at the origin and behaving at least quadratically at zero and at infinity, like for instance the quartic-quadratic potential $V(x)=|x|^2+k|x|^4$ with $k>0$. \hfill$\diamond$
\end{remark}

\subsection{Collapse of the many-body system in the Gross-Pitaevskii ground state}

Next we turn to the $N$-particle quantum Hamiltonian describing our trapped 2D bosons, which reads
\begin{equation}\label{eq:HN}
H_N = \sum_{j=1} ^N  \left( -\Delta_{x_j} + |x_j|^2-2\Omega L_{x_j}\right) - \frac{a}{N-1} \sum_{1\leq i<j \leq N}w_N(x_i-x_j),
\end{equation}
and acts on $\gH^N = \bigotimes_{\rm sym}^N L^2(\R^2)$, the Hilbert space of square-integrable symmetric functions. The two-body interaction $w_N$ approaches a Dirac delta and is chosen in the form 
\begin{align} \label{eq:assumption-wN}
w_N(x)=N^{2\beta} w(N^\beta x)
\end{align}
for a fixed parameter $\beta>0$ and a fixed function $w$ satisfying
\begin{align}
\label{eq:assumption-w1} 
w(x)=w(-x) \geq 0, \quad (1+|x|)w,\ \hat w  \in L^1(\R^2),  \quad   \int_{\R ^2} w = 1.
\end{align}
Finally, $a>0$ is a parameter 
which describes the strength of the interaction. We will take 
$$a=a_N \to a_*$$
which is the Gagliardo-Nirenberg critical constant mentioned before. 

Hamiltonians of the form~\eqref{eq:HN} have generated a huge amount of works in the past decades, in any dimension $d$. The chosen coupling constant proportional to $1/(N-1)$ ensures that the kinetic and interaction energies are comparable in the limit $N\to \infty$. Due to the trapping potential, most of the particles will usually accumulate in a bounded region of space, leading to a high density of order $N$ (in our case they will even collapse at one point). 

In this paper we are interested in the behavior of the ground state energy per particle of~$H_N$,
\begin{equation}\label{eq:GS ener many}
E^{\rm Q}_{\Omega,a}(N) := N^{-1}\inf_{\Psi\in \gH^N, \|\Psi\|=1} \langle \Psi, H_N \Psi \rangle,
\end{equation}
and in the corresponding (non necessarily unique) ground states $\Psi_N$, when $N\to \infty$. In the regime considered in this paper, we expect that the particles will essentially become independent (Bose-Einstein condensation), that is, in terms of wave functions:
\begin{equation}
  \Psi_N(x_1,...,x_N) \approx u^{\otimes N}(x_1,...,x_N):=u(x_1)u(x_2)...u(x_N).
 \label{eq:BEC}
\end{equation}
Indeed, if $w_N\equiv 0$ then the first eigenfunction $\Psi_N$ of $H_N$ is exactly of this form, with $u$ the first eigenfunction of the one-body operator $-\Delta+|x|^2-2\Omega L$. For our interacting Hamiltonian $H_N$, $\Psi_N$ will \emph{never} be of this form, because there is no reason to believe that \emph{all} the particles ought to be in the state $u$. Only of the order of $N$ of them would suffice~\cite{LewNamSerSol-15}. Nevertheless, the ansatz~\eqref{eq:BEC} provides the right energy to leading order, as well as the right density matrices, as we will prove in this paper, and as it has already been shown in many other similar situations, see~\cite{LieSeiSolYng-05,Rougerie-15,Rougerie-hdr} for reviews.

The energy per particle of the fully condensed trial function  $u^{\otimes N}$ is given by the Hartree energy functional
\begin{multline}\label{eq:Hartree func} 
\cE^{\rm H}_{\Omega,a,N}(u)=\frac{\langle u^{\otimes N}, H_N u^{\otimes N}\rangle }{N}= \int_{\R^2} \Big(  |\nabla u(x)|^2 + |x|^2|u(x)|^2\Big)\,dx-2\Omega\pscal{u,Lu}\\-\frac{a}{2}\iint_{\R^2\times\R^2}w_N(x-y) |u(x)|^2 |u(y)|^2\,dx\,dy.
\end{multline}
The infimum of this functional over the set of all $u$'s with $\|u\|_{L^2(\R^2)}=1$,
\begin{equation}
E^{\rm H}_{\Omega,a}(N):=\inf_{\|u\|_{L^2(\R^2)}=1}\cE^{\rm H}_{\Omega,a,N}(u)
\end{equation}
is thus an upper bound to the many-body energy:
$$E^{\rm Q}_{\Omega,a}(N)\leq E^{\rm H}_{\Omega,a}(N).$$ 
When $N\to \infty$, since $w_N \wto \delta_0$, the Hartree functional \eqref{eq:Hartree func} {\em formally} boils down to the trapped nonlinear Gross-Pitaevskii functional $\cE^{\rm GP}_{\Omega,a}$ which we have introduced in~\eqref{eq:def_energy}. We can therefore expect that 
$$E^{\rm H}_{\Omega,a}(N)\simeq E^{\rm GP}_{\Omega}(a),$$
and that their ground states are close. 
At fixed $a<a_*$ this was shown in~\cite{LewNamRou-17b}, but here we need to control the limit $a_N\to a_*$ at the same time as $N\to\ii$ and the corresponding estimates will be provided later in the proof of Proposition~\ref{prop:GSE-1}.

In \cite{LewNamRou-16c,LewNamRou-17b} we have proved that if $a<a_*$ is fixed and $0\leq \Omega<1$, then the many-body ground states of $H_N$ are condensed on the minimizer(s) of the Gross-Pitaevskii functional. In the present paper, we will consider the limit where $a_N\to a_*$ as $N\to \infty$. In that case, the Gross-Pitaevskii minimizer blows up at the center $x=0$ of the trap, as shown in Theorem~\ref{thm:GP} above. We will prove that the many-particle ground state $\Psi_N$ condensates on the exact same function $Q_N$, hence derive a many-body analogue to the result of Guo and Seiringer \cite{GuoSei-14}, at positive rotation.

As usual, the convergence of ground states is formulated using $k$-particles reduced density matrices, defined for any $\Psi_N \in\gH^N$ by a partial trace $$\gamma_{\Psi_N}^{(k)}:= \Tr_{k+1\to N} |\Psi_N \rangle \langle \Psi_N|$$
or, equivalently, $\gamma_{\Psi_N}^{(k)}$ is the trace class operator on $\gH^k$ with kernel 
$$
\gamma_{\Psi_N}^{(k)} (x_1,...,x_k; y_1,...,y_k)= \int_{\R^{2(N-k)}} \Psi_N(x_1,...,x_k,Z) \overline{\Psi_N(y_1,...,y_k,Z)} dZ.
$$  
Bose-Einstein condensation is properly expressed by the convergence in trace norm
$$
\lim_{N\to \infty} \Tr \Big| \gamma_{\Psi_N}^{(k)} - |u^{\otimes k}\rangle \langle u^{\otimes k}| \Big| =0, \quad \forall k\in \mathbb{N}.
$$ 
Our second main result is the following

\begin{theorem}[\textbf{Collapse and condensation of the many-body ground state}]\label{thm:cv-nls}\mbox{}\\
Let $0\leq\Omega<1$, and $a_N=a_*-N^{-\alpha}$ with 
$$0<\alpha< \min\left\{\frac{4}{5}\beta, 2(1-2\beta)\right\}.$$
Let $\Psi_N$ be any ground state of $H_N$. Then we have 
\begin{align} \label{eq:thm-BEC}
\boxed{\lim_{N \to \infty} \Tr\Big| \gamma_{\Psi_{N} }^{(k)} -  |Q_N^{\otimes k} \rangle \langle Q_{N}^{\otimes k}| \Big|=0,}
\end{align}
for all $k\in \mathbb{N}$, where $Q_N$ is the rescaled Gagliardo-Nirenberg optimizer introduced in~\eqref{eq:def_Q_N}.
In addition, we have
\begin{equation}
\boxed{E^{\rm Q}_{\Omega,a_N}(N)= E^{\rm GP}_\Omega(a_N)+o\big(E^{\rm GP}_\Omega(a_N)\big)=\sqrt{a_*-a_N} \left( \frac{2\lambda_*^2}{a_*} + o(1)\right).}
\label{eq:CV_energy}
\end{equation}
\end{theorem}

\bigskip

\begin{remark}\mbox{}\\
(i) Note that the condition $\alpha < 2(1-2\beta)$ implies that we consider mean-field (by opposition to dilute, see~\cite[Section~5.1]{Rougerie-hdr}) interactions, i.e. their range is much larger than the average distance between particles. The latter is set by the length scale  of the GP ground state:
$$ \mbox{distance between particles } \propto \sqrt{\frac{(a_*-a_N)^{1/2}}{N}} \propto N^{-1/2 - \alpha / 4}.$$
We are in fact somewhat deep in the mean-field regime since the transition to dilute interactions would occur when 
$$\mbox{range of the interaction} \propto N^{-\beta} \propto N^{-1/2 - \alpha / 4} \propto \mbox{ distance between particles},$$
i.e. at $\beta \sim 1/2 + \alpha/4$.

The condition $\alpha < 4\beta /5$ is used to ensure that the Hartree and GP ground state problems are close in the limit $N\to \infty$.

\smallskip

(ii) By using the method of~\cite{LewNamRou-17b}, we expect that our result can be extended to a dilute regime as well, i.e. to some (not too large)
$$ \beta > \frac12 + \frac{\alpha}{4}.$$ 
Here we assume $\beta<1/2 - \alpha/4$ for simplicity as this ensures the stability of the many-body system immediately~\cite{LewNamRou-17b}. The approach we follow is significantly simpler than that of~\cite{LewNamRou-17b}, since we use neither the moments estimates, nor the bootstrap on the energy introduced therein. The proof is less flexible however and deeply relies on the uniqueness of the limit profile for GP ground states.  

\smallskip

\noindent (ii) When $a_N=a_*$, it is not clear to us what happens in the large $N$ limit. While the existence of the ground state of $H_N$ still holds true, the blow-up phenomenon becomes more complicated. The behavior of the minimizers for the Hartree functional $\cE^{\rm H}_{\Omega,a_*,N}(u)$ in \eqref{eq:Hartree func} when $N\to \infty$ seems to be open. It seems also difficult to look at the case where $\Omega$ depends on $N$ as well and approaches its limit of stability $\Omega_N\to1$.\hfill$\diamond$ 
\end{remark}

In~\cite{LewNamRou-17b} and several of our previous works~\cite{LewNamRou-14,LewNamRou-16c,LewNamRou-17b,LewNamRou-15,NamRouSei-15}, our approach was based on the \emph{quantum de Finetti theorem}~\cite{Stormer-69,HudMoo-75}, a non-commutative version of the de Finetti-Hewitt-Savage theorem for exchangeable random variables in probability theory~\cite{DeFinetti-31,DeFinetti-37,DiaFre-80,HewSav-55}. More precisely, we used a quantitative version of this theorem in finite-dimensional spaces, which we have proved in~\cite[Lemmas 3.4, 3.6]{LewNamRou-15b} and which extends several previous results by different authors~\cite{KonRen-05,FanVan-06,ChrKonMitRen-07,Chiribella-11,Harrow-13}. The idea of using de Finetti theorems in the context of mean-field limits is not new. For classical systems, this has been put forward by Spohn~\cite{Spohn-81,MesSpo-82,KieSpo-99} and then extended in many directions, see, e.g.,~\cite{CagLioMarPul-92,Kiessling-93,Rougerie-15} and the references therein. For the mean-field limit of quantum systems, the older results in this spirit include~\cite{FanSpoVer-80,VdBLewPul-88,PetRagVer-89,RagWer-89,Werner-92}.

Although one can follow the same strategy here, we give below a different proof of Bose-Einstein condensation, based on a Feynman-Hellman-type argument. This method is much less flexible (it relies on the uniqueness of the limit profile for GP ground states), but it allows to cover a wider range for the parameters $\beta$ and $\alpha$.

\section{Collapse of the rotating GP minimizer: proof of Theorem~\ref{thm:GP}}\label{sec:proof_GP}

In this section we provide the proof of Theorem~\ref{thm:GP}. It is convenient to work at the blow up scale, and thus to rewrite everything in terms of 
$$v(x)=\sqrt\eps\, u(\sqrt{\eps}x),\qquad \eps=\sqrt{a_*-a}.$$
Since the angular momentum $L$ commutes with dilations about the center of rotation, we get
\begin{align*}
\cE^{\rm GP}_{\Omega,a}(u)&=\frac1\eps \int_{\R^2}|\nabla v|^2+\eps\int_{\R^2}|x|^2|v(x)|^2\,dx-2\Omega\pscal{v,Lv}-\frac{a}{2\eps}\int_{\R^2}|v(x)|^4\,dx\\
&=\frac{\cF_{\Omega,\eps}(v)}\eps
\end{align*}
where
\begin{multline}
 \cF_{\Omega,\eps}(v)=\int_{\R^2}|\nabla v|^2-\frac{a_*}2\int_{\R^2}|v(x)|^4\,dx+\eps^2\int_{\R^2}|x|^2|v(x)|^2\,dx\\+\frac{\eps^2}2\int_{\R^2}|v(x)|^4\,dx-2\eps\Omega\pscal{v,Lv}.
\label{eq:def_F}
\end{multline}
We then introduce 
$$\boxed{F_\Omega(\eps)=\inf_{\|v\|_{L^2}=1}\cF_{\Omega,\eps}(v)=\sqrt{a_*-a}\;E^{\rm GP}_\Omega(a)}$$
and our goal is to prove that 
\begin{equation}
 F_\Omega(\eps)=F_0(\eps)+o(\eps^2)=\eps^2 \left( \frac{2\lambda_*^2}{a_*} + o(1)\right).
 \label{eq:reformulation_CV_eps}
\end{equation}
The behavior of $F_0(\eps)$ and its associated unique ground state is studied in~\cite{GuoSei-14}. However, even when $\Omega=0$ we have to prove the convergence of approximate ground states in the sense of~\eqref{eq:def_approx_GS}.

\subsubsection*{\bfseries Step 1. Convergence to a Gagliardo-Nirenberg optimizer.}
By rearrangement inequalities, $F_0(\eps)$ has a radial-decreasing minimizer $\tilde v_\eps$. Then $L\tilde v_\eps=0$, hence
$$F_\Omega(\eps)\leq F_0(\eps)$$
for every $0\leq\Omega<1$. It is the reverse inequality which is not obvious. From the diamagnetic inequality, we have
$$\int_{\R^2}|\nabla v(x)+i\eps\Omega x^\perp v(x)|^2\,dx\geq \int_{\R^2}|\nabla |v|(x)|^2\,dx$$
and therefore we obtain 
\begin{align*}
F_\Omega(\eps)&\geq \min_{\substack{v\in H^1(\R^d)\\ \|v\|_{L^2}=1}}\bigg\{ \int_{\R^2}|\nabla v|^2+(1-\Omega^2)\eps^2\int_{\R^2}|x|^2|v(x)|^2\,dx\\
&\qquad\qquad\qquad -\frac{a_*-\eps^2}2\int_{\R^2}|v(x)|^4\,dx\bigg\}\\
&=\sqrt{1-\Omega^2}\;F_0(\eps).
\end{align*}
This lower bound has the right behavior $O(\eps^2)$ but not the right constant.

We can also write the energy in a different form and obtain the following lower bound
\begin{align}
\cF_{\Omega,\eps}(v)=&(1-\Omega)\int_{\R^2}|\nabla v(x)|^2\,dx+\Omega\int_{\R^2}|\nabla v(x)+i\eps x^\perp v(x)|^2\,dx\nn\\
&-\frac{a_*}2\int_{\R^2}|v(x)|^4\,dx+(1-\Omega)\eps^2\int_{\R^2}|x|^2|v(x)|^2\,dx+\frac{\eps^2}2\int_{\R^2}|v(x)|^4\,dx\nn\\
\geq& \int_{\R^2}|\nabla |v(x)||^2\,dx-\frac{a_*}2\int_{\R^2}|v(x)|^4\,dx\nn\\
&\qquad+(1-\Omega)\eps^2\int_{\R^2}|x|^2|v(x)|^2\,dx+\frac{\eps^2}2\int_{\R^2}|v(x)|^4\,dx\nn\\
\geq& (1-\Omega)\eps^2\int_{\R^2}|x|^2|v(x)|^2\,dx+\frac{\eps^2}2\int_{\R^2}|v(x)|^4\,dx.\label{eq:a_priori_estimate_E_GP}
\end{align}
From these bounds we deduce that any sequence $\{v_\eps\}\subset H^1(\R^2)$ such that $\|v\|_{L^2}=1$ and $\cF_{\Omega,\eps}(v)=O(\eps^2)$ (for instance approximate ground states) is bounded in $H^1(\R^2)$, and that $|x|v_\eps$ is bounded in $L^2(\R^2)$. Such sequences are precompact in $L^p(\R^2)$ for all $2\leq p<\ii$. Therefore, up to a subsequence, we can pass to the limit $v_\eps\to v$ and obtain
\begin{equation}
\int_{\R^2}|\nabla v(x)|^2\,dx-\frac{a_*}2\int_{\R^2}|v(x)|^4\,dx=0\quad\text{with}\quad \int_{\R^2}|v|^2=1.
\label{eq:passing_limit}
\end{equation}
This means that $v$ belongs to the set of the Gagliardo-Nirenberg optimizers (up to a phase)
\begin{equation}
\cQ:=\Big\{Q_{\lambda,X}(x)=\lambda Q_*\big(\lambda(x-X)\big),\quad \lambda>0,\ X\in\R^2\Big\}.
\label{eq:def_cQ}
\end{equation}
Here $Q_*$ is the unique positive radial solution to the equation
$$-\Delta Q_*-a_*Q_*^3=-Q_*.$$
This solution necessarily satisfies $\int_{\R^2}Q_*^2=1$ and it is just given by 
$$Q_*=\|Q\|^{-1}Q=(a_*)^{-1/2}Q$$ 
where $-\Delta Q-Q^3=-Q$. Note that $Q_{\lambda,X}$ solves the equation $-\Delta Q_{\lambda,X}-a_*Q_{\lambda,X}^3=-\lambda^2Q_{\lambda,X}.$

Note that we know from the above arguments that $v_\eps$ converges to $Q_{\lambda,X}$ strongly in $L^2(\R^2)\cap L^4(\R^2)$ and that 
$$ \int_{\R^2}|\nabla v_\eps(x)|^2\,dx-\frac{a_*}2\int_{\R^2}|v_\eps(x)|^4\,dx \to 0.$$
It follows that 
$$ \int_{\R^2}|\nabla v_\eps(x)|^2\,dx \to \frac{a_*}2\int_{\R^2}|Q_{\lambda,X} (x)|^4\,dx = \int_{\R^2}|\nabla Q_{\lambda,X}(x)|^2\,dx $$
and thus that the limit is also strong in $H^1(\R^2)$. For later purposes, we choose the phase of $v_\eps$ such that $v_\eps$ is the closest to its limit:
$$\big\|v_\eps-Q_{\lambda,X}\big\|_{L^2}=\min_\theta \big\|e^{i\theta}v_\eps-Q_{\lambda,X}\big\|_{L^2}.$$
This gives the orthogonality condition on the imaginary part of $v_\eps$:
\begin{equation}
\int_{\R^2}Q_{\lambda,X}\,\mathrm{Im}\big(v_\eps\big)=0.
\label{eq:orthogonality_r}
\end{equation}

\medskip

We have up to now shown that any sequence $\{v_\eps\}$ such that $\cF_\Omega(v_\eps)=O(\eps^2)$ converges to an element of $\cQ$, up to a subsequence and a phase. This is optimal for sequences that have an energy of the order $O(\eps^2)$. In order to determine the possible values of $X$ and $\lambda$, we have to assume that $v_\eps$ is an approximate ground state of $F_\Omega(\eps)$. In the next three steps we actually assume $v_\eps$ is a true ground state, so that we can rely on the variational equation and get better estimates. We return to approximate ground states at the end of the proof.

\subsubsection*{\bfseries Step 2. Decay of ground states.} 
For $v_\eps$ a true ground state, the Euler-Lagrange equation takes the form
\begin{equation}
 -\Delta v_\eps-(a_*-\eps^2)|v_\eps|^2v_\eps+\eps^2|x|^2v_\eps-2\Omega\eps Lv_\eps+\mu_\eps v_\eps=0,
 \label{eq:equation_v_eps}
\end{equation}
with the Lagrange multiplier given by
\begin{equation} \label{eq:mu_eps}
\mu_\eps=-\cF_\Omega(v_\eps)+\frac{a_*-\eps^2}{2}\int_{\R^2}|v_\eps|^4\to \frac{a_*}{2}\int_{\R^2}Q_{\lambda,X}^4=\lambda^2>0.
\end{equation}
Using that $\lambda^2>0$ we shall obtain uniform decay estimates \emph{\`a la} Agmon~\cite{Agmon} for $v_\eps$. First we need to show that $v_\eps$ converges uniformly.

\begin{lemma}[\textbf{Uniform convergence}]\label{lem:H_2}\mbox{}\\
The sequence $\{v_\eps\}$ is bounded in $H^2(\R^2)$ and converges to $Q_{\lambda,X}$ strongly in $H^1(\R^2)$ and in $L^\infty(\R^2)$. 
\end{lemma}

\begin{proof}
We have
$$ \left(-\Delta +\eps^2|x|^2-2\Omega\eps L+\lambda^2\right) v_\eps=(\lambda^2-\mu_\eps) v_\eps+(a_*-\eps^2)|v_\eps|^2v_\eps$$
where the right side is bounded in $L^2(\R^2)$. By the Cauchy-Schwarz inequality and the fact that $L$ commutes with $-\Delta +\eps^2|x|^2$, we have
$$2\eps |L|\leq -\Delta +\eps^2|x|^2$$
and 
$$\left(2\Omega\eps L\right)^2\leq \Omega^2\left(-\Delta +\eps^2|x|^2\right)^2$$
or, equivalently,
\begin{equation}
\norm{2\Omega\eps L\left(-\Delta +\eps^2|x|^2\right)^{-1}}\leq \Omega<1.
\label{eq:control_L}
\end{equation}
By the resolvent formula this proves that 
$$\norm{\left(-\Delta +\eps^2|x|^2\right)\left(-\Delta +\eps^2|x|^2-2\Omega\eps L\right)^{-1}}\leq \frac{1}{1-\Omega}$$
and similarly that
$$\norm{\left(-\Delta +\eps^2|x|^2+\lambda^2\right)\left(-\Delta +\eps^2|x|^2-2\Omega\eps L+\lambda^2\right)^{-1}}\leq \frac{1}{1-\Omega}.$$
Using the relation
\begin{multline*}
\norm{(-\Delta +\eps^2|x|^2+\lambda^2)u}_{L^2(\R^2)}^2=\int_{\R^2}|\Delta u(x)|^2\,dx+\int_{\R^2}(\eps^2|x|^2+\lambda^2)^2|u(x)|^2\,dx\\
+2\int_{\R^2}(\eps^2|x|^2+\lambda^2)|\nabla u(x)|^2\,dx-4\eps^2\int_{\R^2}|u(x)|^2\,dx 
\end{multline*}
we get
$$(-\Delta +\eps^2|x|^2+\lambda^2)^2\geq (-\Delta+\lambda^2)^2-4\eps^2\geq (-\Delta+\lambda^2/2)^2$$
for $\eps\leq\sqrt{3}\lambda/4$, and thus 
$$\norm{\left(-\Delta +\lambda^2/2\right)\left(-\Delta +\eps^2|x|^2-2\Omega\eps L+\lambda^2\right)^{-1}}\leq \frac{1}{1-\Omega}.$$
Inserting in Equation~\eqref{eq:equation_v_eps}, this proves that $v_\eps$ is bounded in $H^2(\R^2)$. Since $v_\eps$ already converges strongly in $L^2(\R^2)$, it also converges strongly in $H^1(\R^2)$ and in $L^\ii(\R^2)$, by interpolation. 
\end{proof}

Next we can prove an exponential decay estimate. We need it later only to obtain the strong convergence $|x|v_\eps\to |x|Q_{\lambda,X}$  in $L^2(\R^2)$.

\begin{lemma}[\textbf{Exponential decay}]\label{lem:decay}\mbox{}\\
The function $v_\eps$ satisfies
$$\int_{\R^2}e^{\lambda|x|}|v_\eps(x)|^2\,dx+\int_{\R^2}|\nabla e^{\frac{\lambda}2|x|}v_\eps|^2\leq C,$$
for a constant $C$ independent of $\eps$. In particular, $|x|v_\eps\to |x|Q_{\lambda,X}$ strongly in $L^2(\R^2)$.
\end{lemma}

\begin{proof}
It is well known that $v_\eps$ is analytic with all its derivatives decaying fast at infinity. We seek here for an explicit bound, independent of $\eps$. We use that
\begin{align*}
-\mathrm{Re}\pscal{v_\eps,e^{\alpha|x|}\Delta v_\eps}&=-\frac12\int_{\R^2}e^{\alpha|x|}\Big(\overline{v_\eps}(x)\Delta v_\eps(x)+v_\eps(x)\Delta \overline{v_\eps}(x)\Big)dx\\
&=-\frac12\int_{\R^2}e^{\alpha|x|}\Big(\Delta|v_\eps|^2(x)-2|\nabla v_\eps(x)|^2\Big)dx\\
&=-\frac12\int_{\R^2}e^{\alpha|x|}\left(\left(\frac{\alpha}{|x|}+\alpha^2\right)|v_\eps(x)|^2-2|\nabla v_\eps(x)|^2\right)dx\\
&=\int_{\R^2}e^{\alpha|x|}|\nabla v_\eps|^2-\frac12\int_{\R^2}e^{\alpha|x|}\left(\frac{\alpha}{|x|}+\alpha^2\right)|v_\eps(x)|^2\,dx\\
&=\int_{\R^2}|\nabla e^{\frac{\alpha}2|x|}v_\eps|^2-\frac{\alpha^2}4\int_{\R^2}e^{\alpha|x|}|v_\eps(x)|^2\,dx.
\end{align*}
Then we integrate the Euler-Lagrange equation~\eqref{eq:equation_v_eps} against $e^{\alpha|x|}\overline{v_\eps}$ and obtain
\begin{align*}
0&=\int_{\R^2}|\nabla e^{\frac{\alpha}2|x|}v_\eps|^2+\int_{\R^2}e^{\alpha|x|}\left(\eps^2|x|^2-(a_*-\eps^2)|v_\eps|^2+\mu_\eps-\frac{\alpha^2}4\right)|v_\eps(x)|^2\,dx\\
&\qquad -2\Omega\eps\pscal{e^{\frac{\alpha}2|x|}v_\eps,Le^{\frac{\alpha}2|x|}v_\eps}\\
&\geq (1-\Omega)\int_{\R^2}|\nabla e^{\frac{\alpha}2|x|}v_\eps(x)|^2\,dx+(1-\Omega)\eps^2\int_{\R^2}e^{\alpha|x|}|x|^2|v_\eps(x)|^2\,dx\\
&\qquad +\int_{\R^2}e^{\alpha|x|}\left(\mu_\eps-(a_*-\eps^2)|v_\eps|^2-\frac{\alpha^2}4\right)|v_\eps(x)|^2\,dx.
\end{align*}
Choosing $\alpha=\lambda$ and using the uniform convergence of $v_\eps$ towards $Q_{\lambda,X}$, we can find a radius $R$ independent of $\eps$ such that 
$$\mu_\eps-(a_*-\eps^2)|v_\eps|^2-\frac{\alpha^2}4\geq\frac{\lambda^2}{2},\qquad \forall\,|x|\geq R,$$
and then 
\begin{equation*}
\frac{\lambda^2}2\int_{\R^2\setminus B_R}e^{\lambda|x|}|v_\eps(x)|^2\,dx+(1-\Omega)\int_{\R^2}\big|\nabla (e^{\frac{\lambda}2|x|}v_\eps)(x)\big|^2\,dx
\leq e^{\lambda R}\left(\mu_\eps+\frac{\lambda^2}{4}+a_*\norm{v_\eps}^2_{L^\ii(B_R)}\right)
\end{equation*}
for all $\eps>0$ small enough. This proves the desired exponential decay estimate. The strong convergence of $|x|v_\eps$ in $L^2(\R^2)$ then follows by interpolation. 
\end{proof}

\subsubsection*{\bfseries Step 3. The imaginary part is (very) small.} 
We split $v_\eps$ into real and imaginary parts
$$ v_\eps = q_\eps + i r_\eps$$
and get bounds on $r_\eps$ using energy estimates (we could similarly use the equation~\eqref{eq:equation_v_eps}). 
Recalling $x^\perp=(-x_2,x_1)$, we observe that
\begin{equation} \label{eq:last-up-a}
\pscal{v_\eps,L v_\eps}= \int_{\R^2}  x^{\perp} \cdot \mathrm{Im} (\overline{v_\eps} \nabla v_\eps) = \int_{\R^2}  x^{\perp} \cdot  (q_\eps \nabla r_\eps-r_\eps \nabla q_\eps) = 2 \int_{\R^2}  x^{\perp} \cdot  q_\eps \nabla r_\eps
\end{equation}
where we have integrated by parts and used that $\mathrm{div}\, x^{\perp} = 0$. Thus 
$$ \left| \pscal{v_\eps,L v_\eps}\right| \leq C \norm{\nabla r_\eps}_{L^2}. $$
Here we have used the fact that $|x|q_\eps$ is bounded in $L^2(\R^2)$. Then the energy reads
$$
\cF_{\Omega,\eps}(v_\eps) \ge \int_{\R ^2} |\nabla q_\eps| ^2 + \int_{\R ^2} |\nabla r_\eps| ^2 - \frac{a_*}{2} \int_{\R^2} \left( q_\eps ^4 +r_\eps^4+   2q_\eps^2 r_\eps ^2 \right)- C \eps \norm{\nabla r_\eps}_{L^2}.
$$
Since $q_\eps\to Q_{\lambda,X}$ and $r_\eps\to 0$ uniformly by Lemma \ref{lem:H_2}, we obtain
$$
\int_{\R^2} |q_\eps ^2 - Q_{\lambda,X} ^2|  r_\eps^2  + \int_{\R^2} r_\eps^4 = o\big(\norm{r_\eps}_{L^2} ^2\big).
$$
Moreover, using the Gagliardo-Nirenberg inequality \eqref{eq:GN} for the real part $q_\eps$, we have 
\begin{align*} 
\int_{\R^2} |\nabla q_\eps|^2 - \frac{a_*}{2}\int_{\R^2}|q_\eps|^4 \ge \left( \int_{\R^2} |\nabla q_\eps|^2 \right) \left( 1- \int_{\R^2}|q_\eps|^2\right) = (\lambda^2+o(1))\int_{\R^2}|r_\eps|^2.
\end{align*}
Here in the second estimate we have used the facts that $\|q_\eps\|_{L^2}^2+\|r_\eps\|_{L^2}^2=1$ and that $q_\eps\to Q_{\lambda,X}$ strongly in $H^1(\R^2)$ by Lemma \ref{lem:H_2}. Thus we can bound the energy from below as  
\begin{equation} \label{eq:Energy-01}
\cF_{\Omega,\eps}(v_\eps) \geq  \int_{\R ^2} |\nabla r_\eps| ^2 -  a_* \int_{\R^2} Q_{\lambda,X}^2 r_\eps ^2 + (\lambda^2+o(1)) \int_{\R^2} |r_\eps|^2  - C \eps \norm{\nabla r_\eps}_{L^2}.
\end{equation}
Now we use some non-degeneracy property of $Q_{\lambda,X}$~\cite{Weinstein-85,McLeod-93,ChaGusNakTsa-08,Frank-13}.
Since $Q_{\lambda,X}$ is positive, it must be the first eigenfunction of the operator 
\begin{equation}
\cL_-:=-\Delta -a_* Q_{\lambda,X} ^2 +\lambda^2
\label{eq:L_+_-}
\end{equation}
and the corresponding eigenvalue $0$ is non-degenerate~\cite[Cor.~11.9]{LieLos-01}. In particular, we get 
\begin{equation}
\pscal{f,\cL_-f}_{L^2(\R ^2)}\geq \lambda_2\norm{f}_{L^2(\R^2)}^2
\label{eq:non-degenerate}
\end{equation}
for all $f$ orthogonal to $Q_{\lambda,X}$ where $\lambda_2>0$ is the second eigenvalue of $\cL_-$. Since on the other hand
\begin{equation}
 \pscal{f,\cL_-f}_{L^2(\R ^2)}\geq \norm{\nabla f}_{L^2(\R^2)}^2-a_*\norm{Q_{\lambda,X}}_{L^\ii(\R^2)}^2\norm{f}_{L^2(\R^2)}^2,
 \label{eq:lower_bound_cL}
\end{equation}
we may combine~\eqref{eq:non-degenerate} with~\eqref{eq:lower_bound_cL} (add a large constant times the second inequality to the first one) and obtain the well known estimate~\cite{Weinstein-85}
\begin{equation}
\pscal{f,\cL_-f}_{L^2(\R ^2)}\geq c\norm{f}_{H^1(\R^2)}^2
 \label{eq:non-degenerate_lower_bound}
\end{equation}
for a constant $c>0$ and all $f$ orthogonal to $Q_{\lambda,X}$. 
Inserting \eqref{eq:non-degenerate_lower_bound} in \eqref{eq:Energy-01} using the fact that $r_\eps$ is orthogonal to $Q_{\lambda,X}$ as we have seen in~\eqref{eq:orthogonality_r}, we obtain 
$$
\cF_{\Omega,\eps}(v_\eps) \geq c_1 \norm{r_\eps}_{H^1}^2 - C \eps \norm{\nabla r_\eps }_{L^2}
 $$
 for a constant $c_1>0$. Combining with the energy upper bound $\cF_{\Omega,\eps}(v_\eps)=O(\eps^2)$ we conclude that 
\begin{equation}\label{eq:bound im}
\norm{r_\eps}_{H^1 (\R^2)} \leq C \eps. 
\end{equation}

\subsubsection*{\bfseries Step 4. Change of gauge and convergence of ground states.} Since $|x|q_\eps$ converges to $|x|Q_{\lambda,X}$ strongly in $L^2$ by Lemma~\ref{lem:decay}, we deduce from \eqref{eq:last-up-a} and ~\eqref{eq:bound im} that 
$$ 
\pscal{v_\eps,L v_\eps}= 2 \int_{\R^2}  x^{\perp} \cdot  Q_{\lambda,X} \nabla r_\eps + o (\eps) = - 2 \int_{\R^2}  x^{\perp} \cdot  \nabla Q_{\lambda,X}  r_\eps + o (\eps).
$$
But $Q_{\lambda,X} = \lambda Q_* (\lambda (x-X))$ with $Q_*$ a radial function, hence
$$ \left(x^{\perp} - X^{\perp}\right) \cdot \nabla Q_{\lambda,X} = 0.$$
Inserting in the above, using~\eqref{eq:bound im} and strong $L^2$-convergence of $|x|q_\eps$ again we obtain   
$$
\pscal{v_\eps,L v_\eps}= 2 \int_{\R^2}  X^{\perp} \cdot  q_\eps \nabla r_\eps + o (\eps).
$$ 
Inserting this in the energy gives 
\begin{multline*}
 \cF_{\Omega,\eps}(v_\eps) = \int_{\R^2} |\nabla v_\eps| ^2 - \frac{a_*}{2} \int_{\R^2} |v_\eps| ^4 + 2 \eps \Omega \int_{\R^2}  X^{\perp} \cdot \mathrm{Im} (v_\eps \nabla \bar{v}_\eps) \\
 + \eps ^2 \left( \frac{1}{2} \int_{\R^2} |v_\eps| ^4 + \int_{\R^2} |x| ^2 |v_\eps| ^2 \right) + o (\eps ^2).
\end{multline*}
Now, define a new function $f_\eps$ by setting  
$$ v_\eps (x)= f_\eps (x) \: e ^{i \eps \Omega X^{\perp} \cdot x}$$
and observe that 
$$ \int_{\R^2} |\nabla v_\eps| ^2 + 2 \eps \Omega \int_{\R^2}  X^{\perp} \cdot \mathrm{Im} (v_\eps \nabla \bar{v}_\eps) = \int_{\R^2} |\nabla f_\eps| ^2 - \eps ^2 \Omega ^2 |X| ^2 \int_{\R^2}  |f_{\eps }|^2.$$
Using the optimal Gagliardo-Nirenberg inequality again and the convergence $f_\eps\to Q_{\lambda,X}$, we obtain 
$$ 
\liminf_{\eps \to 0} \frac{\cF_{\Omega,\eps}(v_\eps)}{\eps ^2} \geq \frac{1}{2} \int_{\R^2} Q_{\lambda,X} ^4 + \int_{\R^2} \left( |x| ^2 - \Omega ^2 |X|^2 \right) Q_{\lambda,X} ^2. 
$$
Recalling that $Q_{\lambda,X}= \lambda Q_* (\lambda (x-X))$ for a radial function $Q_*$ this finally yields
$$ \liminf_{\eps \to 0} \frac{\cF_{\Omega,\eps}(v_\eps)}{\eps ^2}\geq (1-\Omega^2)|X|^2 + \frac{1}{\lambda^2}\int_{\R^2}|x|^2|Q_*(x)|^2\,dx+\frac{\lambda^2}2\int_{\R^2}|Q_*(x)|^4\,dx.$$
Since $\Omega^2<1$, the minimum of the right side is attained for $X=0$ and $\lambda=\lambda_*$ as in~\cite{GuoSei-14}. This also concludes the proof of~\eqref{eq:reformulation_CV_eps}, hence of~\eqref{eq:convergence_energy}. Also, this shows that any sequence of minimizers must, modulo rescaling, choice of a constant phase (in~\eqref{eq:orthogonality_r}) and passing to a subsequence, converge strongly in $H^1$ to $Q_{\lambda^*,0}$. By uniqueness of the limit we conclude that passing to a subsequence is unecessary, which concludes the proof of~\eqref{eq:approx_GS} for true ground states.

\subsubsection*{\bfseries Step 5. Convergence of approximate ground states for $\Omega=0$}
Here we assume $\Omega=0$ and show that any sequence $\{v_\eps\}$ such that 
$$\cF_{0,\eps}(v_\eps)=F_0(\eps)+o(\eps^2)$$
must converge to $Q_{\lambda_*,0}$, as we have proved before for the exact minimizers, using the Euler-Lagrange equation. Indeed, from Step 1 we already know that $v_\eps\to Q_{\lambda,X}$ after extraction of a subsequence and choice of a good phase. Hence we have 
\begin{align*}
 \cF_{0,\eps}(v_\eps)\geq& \eps^2\left(\int_{\R^2}|x|^2|v(x)|^2\,dx+\frac{1}2\int_{\R^2}|v(x)|^4\,dx\right)\\
 \geq& \eps^2\left(\int_{\R^2}|x|^2Q_{\lambda,X}(x)^2\,dx+\frac{1}2\int_{\R^2}Q_{\lambda,X}(x)^4\,dx\right)+o(\eps^2)\\
 =&\eps^2\left(|X|^2+\frac1{\lambda^2}\int_{\R^2}|x|^2Q_*(x)^2\,dx+\frac{\lambda^2}2\int_{\R^2}Q_*(x)^4\,dx\right)+o(\eps^2).
\end{align*}
Again, the minimum of the term in the parenthesis is attained uniquely for $X=0$ and $\lambda=\lambda_*$. Since
$$\cF_{0,\eps}(v_\eps)=F_0(\eps)+o(\eps^2)=\eps^2\left(\frac1{\lambda_*^2}\int_{\R^2}|x|^2Q_*(x)^2\,dx+\frac{\lambda_*^2}2\int_{\R^2}Q_*(x)^4\,dx\right)$$
we must have $X=0$ and $\lambda=\lambda_*$. The limit being unique, the whole sequence must converge to $Q_{\lambda_*,0}$. By usual arguments the limit must be strong in $H^1(\R^2)$.

\subsubsection*{\bfseries Step 6. Convergence of approximate ground states for $0<\Omega<1$}
Next we turn to the rotating case $\Omega<1$. As before we take a sequence $\{v_\eps\}$ such that  $\cF_{\Omega,\eps}(v_\eps)=F_\Omega(\eps)+o(\eps^2)=F_0(\eps)+o(\eps^2)$.
We artificially increase the rotation speed by choosing a $\eta<1$ be such that $\Omega/\eta<1$ and we remark that
\begin{align*}
F_\Omega(\eps)+o(\eps^2)=\cF_{\Omega,\eps}(v_\eps)&=\eta\cF_{\Omega/\eta,\eps}(v_\eps)+(1-\eta)\cF_{0,\eps}(v_\eps)\\
&\geq\eta F_{\Omega/\eta}(\eps)+(1-\eta)F_{0}(\eps).
\end{align*}
Since the terms $\eps^{-2}F_\Omega(\eps)$, $\eps^{-2}F_{\Omega/\eta}(\eps)$ and $\eps^{-2}F_0(\eps)$ all have the same limit, this proves that 
$\cF_{\Omega/\eta,\eps}(v_\eps)=F_{\Omega/\eta}(\eps)+o(\eps^2)$ and $\cF_{0,\eps}(v_\eps)=F_{0}(\eps)+o(\eps^2)$. But then $\{v_\eps\}$ is a sequence of approximate ground states  in the case $\Omega=0$ and we can apply the previous step. 

\medskip

This concludes the proof of Theorem~\ref{thm:GP}.\qed

\section{Collapse of the many-body ground state: proof of Theorem~\ref{thm:cv-nls}} 

This section is devoted to the proof of Theorem~\ref{thm:cv-nls}. We start with the convergence of the ground state energy, and then settle some energy estimates for the ground state.

\subsubsection*{\bfseries Step 1. Convergence of the many-body ground state energy}
We provide the proof of the convergence of $E^{\rm Q}_{\Omega,a_N}(N)$, using a method described in~\cite[Section~3]{Lewin-15}. The precise statement is the following. 

\begin{proposition}[\textbf{Convergence of the many-body ground state energy}] \label{prop:GSE-1}\mbox{}\\
Let $\beta\in (0,1/2)$ and $a_*-a_N = N^{-\alpha}$ with 
\begin{equation}
0<\alpha< \min\Big\{\frac{4}{5}\beta,2(1-2\beta)\Big\}.
\label{eq:better_condition_alpha}
\end{equation}
Then 
\begin{align}\label{eq:eN}
E^{\rm Q}_{\Omega,a_N}(N)=E^{\rm GP}_\Omega(a_N)+o\big(E^{\rm GP}_\Omega(a_N)\big)=  \sqrt{a_*-a_N} \left( \frac{2\lambda^2}{a_*}+ o(1)\right).
\end{align}
\end{proposition}

\begin{proof} We start with the lower bound. From the arguments in~\cite[Section 3]{Lewin-15} and the fact that the Fourier transform of $w_N$ satisfies $|\widehat w_N|\le CN^{2\beta}$ we have
\begin{multline}
E^{\rm Q}_{\Omega,a}(N) \ge \min_{\substack{\gamma=\gamma_*\geq0\\ \tr\gamma=1}}\left\{ \tr\left[ (-\Delta+|x|^2-2\Omega L)\gamma \right]-\frac{a}2\int_{\R^d}\int_{\R^d}w_N(x-y)\rho_\gamma(x)\rho_\gamma(y)\,dx\,dy\right\}\\ - CN^{2\beta-1}.
\label{eq:ineq_boltzons}
\end{multline}
Here the one-body density is defined by writing 
$$\rho_\gamma=\sum_j n_j|u_j|^2$$
with the spectral decompostion $\gamma=\sum_j n_j|u_j\rangle\langle u_j|$ ($0\leq n_j\leq1$ and $\sum_j n_j=1$).

Note that the arguments in~\cite[Section 3]{Lewin-15} contain two ingredients. One is to get a lower bound involving the mean-field interaction following an idea from~\cite{LevyLeblond-69,LieYau-87}, using auxiliary classical particles that repel each other in order to model the attractive part of the interaction.  The other ingredient is to apply the Hoffmann-Ostenhof inequality\footnote{Bounding the full kinetic energy from below by that of the one-body density $\sqrt{\rho_{\Psi_N}}$.} to the kinetic energy in order to get the Hartree energy. Here we cannot use the second part when $\Omega\neq0$. We thus bypass it and only bound the interaction from below. The price to pay is that we end up with the mixed state type Hartree energy on the right side of~\eqref{eq:ineq_boltzons}. 

By the Cauchy-Schwarz inequality, we have
\begin{multline}
\iint |u(x)|^2|u(y)|^2 w_N(x-y) dx dy \leq \frac{1}{2}\iint (|u(x)|^4+ |u(y)|^4) w_N(x-y) dx dy \\
= \int |u(x)|^4 dx
\label{eq:estim_interaction_delta}
\end{multline}
for any $u\in H^1(\R^2)$. Applying this to $u=\sqrt{\rho_\gamma}$, we therefore get 
\begin{equation}
E^{\rm Q}_{\Omega,a}(N) \ge \min_{\substack{\gamma=\gamma_*\geq0\\ \tr\gamma=1}}\left\{\tr\left[(-\Delta+|x|^2-2\Omega L)\gamma\right] -\frac{a}2\int_{\R^d}\rho_\gamma(x)^2\,dx\right\} - CN^{2\beta-1}.
\label{eq:ineq_boltzons_bis}
\end{equation}
Here $\rho_\gamma(x):=\gamma(x,x)$ for every $\gamma=\gamma^*\geq0$ with $\tr(\gamma)=1$. The usual Gross-Pitaevskii energy~\eqref{eq:def_energy} corresponds to $\gamma=|u\rangle\langle u|$, a rank one projection. Now we remark that the minimum on the right of~\eqref{eq:ineq_boltzons_bis} is nothing but the original Gross-Pitaevskii minimum.

\begin{lemma}[\textbf{Mixed Gross-Pitaevskii energy}]\label{lem:mixed_GP}\mbox{}\\
For every $0\leq\Omega<1$ and every $0\leq a<a_*$, we have
\begin{equation}
\min_{\substack{\gamma=\gamma_*\geq0\\ \tr\gamma=1}}\left\{\tr\left[(-\Delta+|x|^2-2\Omega L)\gamma\right] -\frac{a}2\int_{\R^d}\rho_\gamma(x)^2\,dx\right\}=E^{\rm GP}_{\Omega}(a).
\end{equation}
\end{lemma}

\begin{proof}[Proof of Lemma~\ref{lem:mixed_GP}]
This follows from the concavity of the energy with respect to $\gamma$. Indeed, writing $\gamma=\sum_j n_j|u_j\rangle\langle u_j|$ with $0\leq n_j\leq1$ and $\sum_j n_j=1$, we have $\rho_\gamma=\sum_j n_j|u_j|^2$, hence
$$\tr\left[(-\Delta+|x|^2-2\Omega L)\gamma\right] -\frac{a}2\int_{\R^d}\rho_\gamma(x)^2\,dx\geq \sum_j n_j\,\cE^{\rm GP}_{\Omega,a}(u_j)\geq E^{\rm GP}_\Omega(a)$$
as we wanted.
\end{proof}

Inserting in~\eqref{eq:ineq_boltzons_bis} and using the behavior of $E^{\rm GP}_\Omega(a_N)$ proved in Theorem~\ref{thm:GP}, we get the simple lower bound
$$\boxed{E^{\rm Q}_{\Omega,a_N}(N) \ge  E^{\rm GP}_\Omega(a_N)\left(1-CN^{2\beta-1}(a_*-a_N)^{-1/2}\right).}$$
The error term $N^{2\beta-1}(a_*-a_N)^{-1/2}$ goes to zero when $a_*-a_N=N^{-\alpha}$ and $\alpha<2(1-2\beta)$.

Now we turn to the upper bound. By the variational principle 
$$
E^{\rm Q}_{\Omega,a}(N) = \inf_{\Psi\in \gH^N, \|\Psi\|=1} \frac{ \langle \Psi, H_N \Psi \rangle}{N} \le  \inf_{\|u\|_{L^2}=1} \frac{ \langle u^{\otimes N}, H_N u^{\otimes N}\rangle}{N} = E^{\rm H}_{\Omega,a}(N).
$$
We therefore need to bound the Hartree functional from above by the Gross-Pitaevskii functional. Introducing the variable $z=N^\beta(x-y)$, we may write
\begin{align*}
&\iint_{\R^2\times\R^2} |u(x)|^2 N^{2\beta}w(N^\beta (x-y)) |u(y)|^2 dxdy - \left( \int_{\R^2} w\right) \int_{\R^2} |u(x)|^4dx \\ 
&\qquad=  \iint_{\R^2\times\R^2} |u(x)|^2  w(z) \Big( |u(x-N^{-\beta}z)|^2 - |u(x)|^2 \Big) dxdz \\
&\qquad= \iint_{\R^2\times\R^2}  |u(x)|^2  w(z) \Big(  \int_{0}^1 (\nabla |u|^2)(x-tN^{-\beta}z) \cdot ( N^{-\beta} z) dt \Big) dxdz.
\end{align*} 
Using $|\nabla |u|^2| \le 2 |\nabla u|.|u|$ and H\"older's inequality, we find
\begin{align*}
&\int_{\R^2} |u(x)|^2 |(\nabla |u|^2)(x-tN^{-\beta}z)| dx\\
&\qquad\le 2\Big(\int_{\R^2} |u(x)|^6 dx \Big)^{1/3}\Big(\int_{\R^2} |\nabla u(x-tN^{-\beta}z)|^2 dx \Big)^{1/2}\Big(\int_{\R^2} |u(x-tN^{-\beta}z)|^6 dx \Big)^{1/6} \\
& \qquad= 2 \|\nabla u\|_{L^2(\R^2)} \|u\|_{L^6(\R^2)}^3.
\end{align*}
Thus 
\begin{align*}
|\cE^{\rm H}_{\Omega,a_N,N}(u)-\cE^{\rm GP}_{\Omega,a_N}(u)|=&\frac{a_N}{2}\left| \iint_{\R^2\times\R^2} |u(x)|^2 w_N(x-y) -   \int_{\R^2} |u(x)|^4dx \right| \\ 
=& \,\frac{a_N}{2} \left| \iint_{\R^2\times\R^2}  |u(x)|^2  w(z) \Big(  \int_{0}^1 (\nabla |u|^2)(x-tN^{-\beta}z) \cdot ( N^{-\beta} z) dt \Big) dxdz \right| \\
\le&  a_N N^{-\beta}\Big(\int_{\R^2} |w(z)|\;|z| dz \Big) \|\nabla u\|_{L^2(\R^2)} \|u\|_{L^6(\R^2)}^3\\
\le& C N^{-\beta} \|\nabla u\|_{L^2(\R^2)} \|u\|_{L^6(\R^2)}^3.
 \end{align*} 
 Now choosing the trial state $Q_{N}$ as in \eqref{eq:def_Q_N} we find that
\begin{align*}
\cE^{\rm H}_{\Omega,a_N,N}(Q_N) &\le \cE^{\rm GP}_{\Omega,a_N}(Q_N) + C N^{-\beta} \|\nabla Q_{N}\|_{L^2} \|Q_{N}\|_{L^6}^3\\
&= \sqrt{a_*-a_N} \left( \frac{2\lambda^2}{a_*} + CN^{-\beta} (a_*-a_N)^{-\frac{5}{4}}\right).
\end{align*}
The error term goes to zero when $a_*-a_N \ge N^{-\alpha}$ with $\alpha< 4\beta/5$. We have proved the upper bound
$$\boxed{E^{\rm Q}_{\Omega,a_N}(N) \le  E^{\rm GP}_\Omega(a_N)\left(1+CN^{-\beta} (a_*-a_N)^{-\frac{5}{4}}\right)}$$
which concludes the proof of Proposition~\ref{prop:GSE-1}.
\end{proof}

\subsubsection*{\bfseries Step 2. Convergence of reduced density matrices}
This is a Feynman-Hellmann-type argument. 
Let $\eta >0$ be a small parameter to be fixed later on and $A$ be a bounded self-adjoint operator on $L^2 (\R^2)$. Consider the perturbed Hamiltonian
\begin{equation}\label{eq:HN eta}
H_{N,\eta} = \sum_{j=1} ^N  \left( -\Delta_{x_j} + |x_j|^2-2\Omega L_{x_j} + \eta A_j \right) - \frac{a}{N-1} \sum_{1\leq i<j \leq N}w_N(x_i-x_j)
\end{equation}
with ground-state energy per particle denoted $E^{\rm Q}_\eta$ hereafter. Introduce the associated Gross-Pitaevskii energy functional (we drop some $\Omega,a$ subscrits during this proof, for lightness of notation)
\begin{equation}
\cE^{\rm GP}_{\eta}(u)=\int_{\R^2}|\nabla u(x)|^2\,dx+\int_{\R^2}|x|^2|u(x)|^2\,dx-2\Omega\pscal{u,Lu} + \eta\pscal{u,Au} -\frac{a}2\int_{\R^2}|u(x)|^4\,dx.
\label{eq:def_energy eta}
\end{equation}
In what follows we denote by $u_\eta$ a ground state for the latter and $E^{\rm GP}_\eta=\cE^{\rm GP}_\eta(u_\eta)$ the corresponding ground-state energy.

Let $\Psi_N$ be a ground state for $H_N = H_{N,0}$ and $\gamma_{\Psi_N} ^{(1)}$ its one-body reduced density matrix. We write 
\begin{align}\label{eq:variation}
\eta \tr\left[ A \gamma_{\Psi_N} ^{(1)} \right] &= N^{-1}\big(\left\langle \Psi_N | H_{N,\eta} | \Psi_N \right\rangle - \left\langle \Psi_N | H_{N,0} | \Psi_N \right\rangle \big)\nonumber\\
&\geq E^{\rm Q}_\eta - E^{\rm Q}_0\nn\\
&\geq E^{\rm GP}_\eta - E^{\rm GP}_0 + O (N^{2\beta -1}) + O (N^{3\alpha/4 - \beta})\nonumber\\
&\geq \cE ^{\rm GP}_\eta (u_\eta) - \cE ^{\rm GP}_0 (u_\eta) +  O (N^{2\beta -1}) + O (N^{3\alpha/4 - \beta})\nonumber\\
&= \eta \left\langle u_\eta | A | u_\eta \right\rangle +  O (N^{2\beta -1}) + O (N^{3\alpha/4 - \beta}).
\end{align}
The first inequality is the variational principle, the second uses the estimates of the previous step. In that regard, observe that the energy lower bound applies \emph{mutatis mutandis} to the problem perturbed by $\eta A$. The error term in \eqref{eq:ineq_boltzons} solely comes from applying the L\'evy-Leblond method to the interaction as in~\cite[Section~3]{Lewin-15}. It is therefore independent of the one-body term (in particular, of $\eta$ and $A$). Lemma~\ref{lem:mixed_GP} generalizes to the perturbed functional, for the only property of the one-body energy used in its proof is its linearity in the density matrix. The third inequality in~\eqref{eq:variation} is the variational principle again.

Under the assumption that 
$$ \alpha < \min \left( \frac{4\beta}{5}, 2(1-2\beta) \right)$$
one can pick some $\eta=\eta_N\to0$ as $N\to\infty$, such that 
$$ \eta ^{-1} N^{2\beta -1} + \eta ^{-1} N^{3\alpha/4 - \beta} \underset{N\to\infty}{\longrightarrow} 0$$
and also 
$$ \eta =o\big(E^{\mathrm GP}\big)=o\big( \sqrt{a_* - a_N}\big) = o\big(N ^{-\alpha/2}\big).$$
Then, dividing~\eqref{eq:variation} by $\eta$ and repeating the argument with $A$ changed to $-A$ yields 
\begin{equation}\label{eq:almost}
 \left\langle u_\eta |A | u_\eta \right\rangle + o (1)\leq \tr\left[ A \gamma_{\Psi_N} ^{(1)}\right] \leq  \left\langle u_{-\eta} |A | u_{-\eta }\right\rangle + o (1). 
\end{equation}
On the other hand, with the above choice of $\eta$, since 
$$ \cE ^{\rm GP}_0 (u_\eta) = \cE ^{\rm GP}_\eta (u_\eta) + O(\eta \norm{A}) \leq \cE ^{\rm GP}_\eta (u_0) + O(\eta \norm{A}) = E ^{\rm GP}_0 + O(\eta \norm{A}),$$
it follows that $(u_\eta)$ and $(u_{-\eta})$ are sequences of quasi-minimizers for $\cE^{\rm GP}_0$. We may apply Theorem~\ref{thm:GP} to them, and thus both sequences satisfy~\eqref{eq:approx_GS}. Combining with~\eqref{eq:almost}, we get, after a dilation of space variables,  trace-class weak-$\star$ convergence of $\gamma_{\Psi_N} ^{(1)}$ to $|Q_N\rangle \langle Q_N|$. Since no mass is lost in the limit, the convergence must hold in trace-class norm, which gives~\eqref{eq:thm-BEC} for $k=1$. 

To obtain~\eqref{eq:thm-BEC} for $k > 1$, observe that, after dilation, $\gamma_{\Psi_N} ^{(1)}$ converges in trace-class norm to a rank-one operator. It is well-known that this implies convergence of higher order density matrices to tensor powers of the limiting operator (see e.g. the discussion following~\cite[Theorem~7.1]{LieSeiSolYng-05} or~\cite[Corollary~2.4]{LewNamRou-14}). \hfill\qed

\begin{remark}\mbox{}\\
Note that we have been able to obtain the convergence of density matrices from that of the energy by a rather soft argument. What makes this possible is Lemma~\ref{lem:mixed_GP} and the uniqueness of (the limit of) the GP minimizer. Lemma~\ref{lem:mixed_GP} relies strongly on the fact that the interaction is attractive. For repulsive interactions\footnote{There is no blow-up then, but one might want to adapt the method to obtain convergence of density matrices to the stable GP minimizers.}, the argument can be adapted provided the one-body part is positivity-preserving (hence, without rotation), using that the bosonic and boltzonic minimization problems then coincide~\cite[Theorem~3.3]{LieSei-09}.\hfill$\diamond$
\end{remark}

\appendix
\section{Extension to anharmonic potentials}\label{sec:extensions}

The arguments given in this paper can be extended in various directions. One possibility is to consider anharmonic potentials. For completeness we state here the corresponding result when the external potential is chosen in the form
$$V(x)=c_0|x|^s$$
but we expect similar results when $V$ has a unique minimizer and behaves like this in a neighborhood of this point, similarly to what was done in~\cite{GuoSei-14}.
When $s\neq2$ the limit $a\to a_*$ requires to have $\Omega=0$. Although a stronger confinement $s>2$ can control the rotating gas at infinity, it is not sufficient to control rotating effects near the blow up point. So we do not consider any rotation here. The many-particle Hamiltonian then takes the form
\begin{equation}\label{eq:HN_s}
\tilde H_N = \sum_{j=1} ^N  \left( -\Delta_{x_j} + c_0|x_j|^s\right) - \frac{a}{N-1} \sum_{1\leq i<j \leq N}w_N(x_i-x_j).
\end{equation}
The following can be proved by arguing exactly as we did for $s=2$.

\begin{theorem}[\textbf{Collapse and condensation of the many-body ground state for anharmonic potentials}]\label{thm:cv-nls_s}\mbox{}\\
Let $\Omega\equiv0$, $s>0$, $c_0>0$, $0<\beta<1/2$ and $a_N=a_*-N^{-\alpha}$ with 
$$0<\alpha< \min\left\{\frac{ \beta (s+2)}{s+3}, \frac{(1-2\beta)(s+2)}{s}\right\}.$$
Let $\Psi_N$ be the unique ground state of $\tilde H_N$. Then we have 
\begin{align} \label{eq:thm-BEC_s}
\lim_{N \to \infty} \Tr\Big| \gamma_{\Psi_{N} }^{(k)} -  |\tilde Q_N^{\otimes k} \rangle \langle \tilde Q_{N}^{\otimes k}| \Big|=0,
\end{align}
for all $k\in \mathbb{N}$, where $\tilde Q_N$ is the rescaled Gagliardo-Nirenberg optimizer given by
$$\tilde Q_N(x)= (a_*)^{-1/2}\tilde\lambda (a_*-a_N)^{-\frac{1}{2+s}} Q\left( \tilde\lambda (a_*-a_N)^{-\frac{1}{2+s}} x\right),$$
with
$$\tilde\lambda=\left( \frac{s}{2} c_0\int_{\R^2} |x|^s |Q(x)|^2 dx \right)^{\frac{1}{2+s}}.$$
In addition, we have
\begin{equation}
\frac{\min\sigma(\tilde H_N)}{N}= (a_*-a_N)^{\frac{s}{s+2}} \Big( \frac{\tilde{\lambda}^2}{a_*}\frac{s+2}{s} + o(1)\Big).
\label{eq:CV_energy_s}
\end{equation}
\end{theorem}
The right side of~\eqref{eq:CV_energy_s} is of course the expansion of the Gross-Pitaevskii energy, which has already been derived in~\cite{GuoSei-14}.


\end{document}